\documentclass[12pt]{amsart}
\usepackage{times}
\usepackage{verbatim,amssymb,latexsym,amscd}
\usepackage{bbold}
\usepackage[all]{xy}
\usepackage[colorlinks,linkcolor=blue,citecolor=blue,urlcolor=red]{hyperref}
\usepackage[T1]{fontenc}

\newcommand{\ie}{{\it i.e. }}

\newcommand{\resp}{{\it resp. }}

\newcommand{\A}{\mathbf{A}}

\newcommand{\G}{\mathbb{G}}

\renewcommand{\L}{\mathbb{L}}

\newcommand{\Q}{\mathbf{Q}}

\newcommand{\Z}{\mathbf{Z}}

\newcommand{\sF}{\mathcal{F}}
\newcommand{\sG}{\mathcal{G}}
\newcommand{\sH}{\mathcal{H}}

\newcommand{\sK}{\mathcal{K}}

\newcommand{\sM}{\mathcal{M}}

\newcommand{\sO}{\mathcal{O}}

\newcommand{\sS}{\mathcal{S}}

\newcommand{\Ker}{\operatorname{Ker}}
\newcommand{\Coker}{\operatorname{Coker}}

\newcommand{\Hom}{\operatorname{Hom}}

\newcommand{\uHom}{\operatorname{\underline{Hom}}}
\newcommand{\Ext}{\operatorname{Ext}}

\newcommand{\HI}{{\operatorname{\bf HI}}}

\newcommand{\tors}{{\operatorname{tors}}}
\newcommand{\rat}{{\operatorname{rat}}}

\newcommand{\num}{{\operatorname{num}}}

\newcommand{\nr}{{\operatorname{nr}}}

\newcommand{\DM}{\operatorname{\bf DM}}

\newcommand{\cont}{{\operatorname{cont}}}

\newcommand{\et}{{\text{\rm \'et}}}
\renewcommand{\o}{{\text{\rm o}}}

\newcommand{\eff}{{\text{\rm eff}}}

\newcommand{\rd}{{\text{\rm rd}}}
\newcommand{\nrd}{{\text{$n$-\rm rd}}}

\newcommand{\by}{\xrightarrow}

\newcommand{\iso}{\by{\sim}}

\newcommand{\inj}{\hookrightarrow}

\newcommand{\surj}{\rightarrow\!\!\!\!\!\rightarrow}
\newcommand{\Surj}{\relbar\joinrel\surj} 
\newcommand{\colim}{\varinjlim}
\renewcommand{\lim}{\varprojlim}

\renewcommand{\qed}{\hfill $\Box$\medskip}

\renewcommand{\phi}{\varphi}
\renewcommand{\epsilon}{\varepsilon}

\newcounter{spec}
{\end{list}}%

\numberwithin{equation}{section}

\newtheorem{thm}{Theorem}[section]
\newtheorem{lemma}[thm]{Lemma}
\newtheorem{prop}[thm]{Proposition}
\newtheorem{cor}[thm]{Corollary}

\theoremstyle{definition}

\newtheorem{defn}[thm]{Definition}

\newtheorem{rk}[thm]{Remark}

\newtheorem{ex}[thm]{Example}

\begin{document}
\title{An $l$-adic norm residue epimorphism theorem}
\author{Bruno Kahn}
\address{CNRS, Sorbonne Université and Université Paris Cité, IMJ-PRG\\ Case 247\\4 place
Jussieu\\75252 Paris Cedex 05\\France}
\email{bruno.kahn@imj-prg.fr}
\begin{abstract} We show that the continuous étale cohomology groups \allowbreak $H^n_\cont(X,\Z_l(n))$ of smooth varieties $X$ over a finite field $k$ are spanned as $\Z_l$-modules by the $n$-th Milnor $K$-sheaf locally for the Zariski topology, for all $n\ge 0$. Here $l$ is a prime invertible in $k$. This is the first general unconditional result towards the conjectures of \cite{tate} which put together the Tate and the Beilinson conjectures relative to algebraic cycles on smooth projective $k$-varieties.
\end{abstract}
\date{October 9, 2024}
\maketitle

\section{Introduction}

Two fundamental conjectures on smooth projective varieties $X$ over a finite field $k$ are
\begin{itemize}
\item the Tate conjecture: for any $n\ge 0$, the order of the pole of the zeta function $\zeta(X,s)$ at $s=n$ equals the rank of the group of algebraic cycles of codimension $n$ over $X$, modulo numerical equivalence;
\item the Beilinson conjecture: for any $n\ge 0$, an algebraic cycle of codimension $n$ on $X$ with $\Q$-coefficients which is numerically equivalent to $0$ is rationally equivalent to $0$.
\end{itemize}

In the unpublished preprint \cite{tate}, I put these two conjectures together and reformulated them into a sheaf-theoretic statement involving all smooth (not necessarily projective) $k$-varieties. 

Actually, there are two reformulations in \cite{tate}: one with rational coefficients (Conjecture 8.12) and one with integral coefficients (Conjecture 9.6). The first one is elementary, involving cohomology of Milnor $K$-sheaves; the second one involves  motivic cohomology and also appears in the published paper \cite{glr} (Conj. 3.2 and Th. 3.4).

Here we shall be interested in the first reformulation. Let me recall it. Let $\sS$ denote the étale site of smooth $k$-varieties; as in \cite[Def. 2.1]{tate}, write $\Z_l(n)^c$ (\resp $\Q_l(n)^c$) for the object
$R\lim(\mu_{l^\nu}^{\otimes n})$ (\resp $\Z_l(n)^c\otimes \Q$) of
$D^+(Ab(\sS))$. Thus,
\[H^i_\et(X,\Z_l(n)^c)=H^i_\cont(X,\Z_l(n))\]
where the right hand side is Jannsen's continuous étale cohomology \cite{jannsen}, and similarly with $\Q_l$ coefficients.

As a first step, we have
\begin{equation}\label{eq0}
H^i(\Q_l(n)^c)=0 \text{ for } i<n 
\end{equation}
by \cite[Cor. 6.10 a)]{tate}. In particular, the presheaf $X\mapsto H^n_\cont(X,\Q_l(n))$ is already an étale sheaf. Then, by  \cite[Prop. 8.10 and its proof]{tate}, a version of a theorem of Tate \cite[Th. 3.1]{Tate} yields a homomorphism
\begin{equation}\label{eq-3}
K_n^M(X)\otimes \Z_l\to H^n_\cont(X,\Z_l(n))
\end{equation}
for any smooth $X$, where $K_n^M(X):=K_n^M(\Gamma(X,\sO_X))$; 
hence a homomorphism of associated Zariski sheaves
\begin{equation}\label{eq-4}
\sK_n^M\otimes \Z_l\to \sH^n(\Z_l(n)^c).
\end{equation}
and a fortiori a homomorphism of associated étale sheaves
\begin{equation}\label{eq-1}
\alpha^*\sK_n^M\otimes \Z_l\to \alpha^*\sH^n(\Z_l(n)^c)
\end{equation}
where $\alpha$ is the projection of $\sS$ on the big smooth Zariski site. 
By \eqref{eq0}, we then get a morphism in $D^+(\sS)$
\[\alpha^*\sK_n^M[-n]\otimes\Q_l\to\Q_l(n)^c.\]

For $n=0$, this morphism is not an isomorphism because the right hand side has two nonzero cohomology sheaves, coming from $H^0(k,\Q_l)$ and $H^1(k,\Q_l)$ \cite[Th. 4.6 b)]{tate}. To get the correct comparison morphism, we tensor with it to get
\begin{equation}\label{eq33} \alpha^*\sK_n^M[-n]\otimes^L \Q_l(0)^c\to
\Q_l(n)^c.
\end{equation}

Conjecture 8.12 of \cite{tate} states that \eqref{eq33} is an isomorphism. Note that $\Q_l(0)^c\simeq \Q_l\oplus \Q_l[-1]$ by \cite[Cor. 4.5 and Th. 4.6]{tate}. In concrete terms, \eqref{eq33} therefore induces homomorphisms
\[H^{i-n-1}(X,\sK_n^M)\otimes \Q_l\oplus H^{i-n}(X,\sK_n^M)\otimes \Q_l\to H^i_\cont(X,\Q_l(n))\]
for any smooth $X$, where the right group is Jannsen's continuous étale cohomology \cite{jannsen}, and Conjecture 8.12 predicts that they are isomorphisms. This can be viewed as an extension of the cohomological version of Tate's conjecture  saying that, in some sense, all continuous étale cohomology groups are generated by ``algebraic cycles'' (cohomology of Milnor $K$-sheaves: note that $H^n(X,\sK_n^M)\simeq CH^n(X)$), plus one transcendental element: the generator of $H^1_\cont(k,\Q_l)=\Hom_\cont(G_k,\Q_l)$ which sends Frobenius to $1$.

As a special case, this conjecture proposes a description of the first non-zero continuous étale cohomology group $H^n_\cont(X,\Q_l(n))$, which ought to be isomorphic to $H^0(X,\sK_n^M)\otimes \Q_l$ via \eqref{eq-1}. I realised recently that a refinement of the proof of \eqref{eq0} might give enough information on this group to approach this latter conjecture. This was successful, and we even get an integral statement which is the main result of this article:

\begin{thm}\label{t0} 
The morphism \eqref{eq-1} is an epimorphism of Zariski sheav\-es, and even of presheaves if $n\le 2$. 
\end{thm} 

This is the first general unconditional result in the direction of \cite[Conj. 8.12]{tate}. It can be viewed as an $l$-adic norm residue epimorphism theorem. As a complement, let us notice that the Zariski and étale sections of both sides coincide by \cite[Th. 14.24 and 22.2]{mvw}, and that, after tensoring with $\Q$, those of the right hand side on some smooth $X$ are $H^n_\cont(X,\Q_l(n))$ by \eqref{eq0}. So, for $n\le 2$, Theorem \ref{t0} yields a surjection
\[H^0(X,\sK_n)\otimes \Q_l\Surj H^n_\cont(X,\Q_l(n))\]
for all smooth $X$.

To avoid any misunderstanding, let me point out that Theorem \ref{t0} is \emph{deduced} from its rational version, Theorem \ref{t0Q}, by using the norm residue isomorphism theorem of \cite{voem}, while the proof of Theorem \ref{t0Q} itself has nothing to do with the latter theorem.


One may ask about isomorphy in Theorem \ref{t0Q}. But, after tensorisation with $\Q$, the global sections of the right hand side of \eqref{eq-1} are $0$ on $X$ if $X$ is projective (provided $n>0$), so this would imply the vanishing of $H^0(X,\sK_n^M)\otimes \Q$ for such $X$. Conversely, this vanishing for all smooth projective varieties would imply that \eqref{eq-1}$\otimes \Q$ is an isomorphism: see beginning of Section \ref{s7}. It can actually be proven for certain smooth projective $X$'s (Theorem \ref{t6}), but there aren't enough of them to deduce the isomorphy of \eqref{eq-1} in general. See nevertheless Corollary \ref{c3} and Example \ref{ex1}  for examples.

The proof of Theorem \ref{t0Q} is not difficult, but involves a number of ideas. Here is a description. By de Jong's theorem on alterations, we reduce to the case where $X$ has a smooth compactification whose closed complement is the support of a divisor with strict normal crossings. A suitable spectral sequence, plus cohomological purity, then allows us to get a concrete description of $H^n_\cont(X,\Q_l(n))$, as in Corollary \ref{c1} b). This description already shows that these cohomology classes are, in some sense, of an algebraic nature, and the next step is to make the link with \eqref{eq-1}. Here we pass to Voevodsky's theory of homotopy invariant Nisnevich sheaves with transfers \cite{mvw} and its extension by Déglise to \emph{homotopic modules} \cite{deglise}. It turns out that the collection of the $H^n_\cont(X,\Q_l(n))$'s, for $n\in \Z$, defines a special kind of homotopic module that we call \emph{reduced} (Definition \ref{d4}; see Proposition \ref{p5}). The Milnor $K$-sheaves, on their part, form a homotopic module which maps to the latter via \eqref{eq-1}, but this homotopic module is not known to be reduced (this is precisely the vanishing issue mentioned in the previous paragraph). However, any ($(-1)$-connected) homotopic module admits a universal map to a reduced one, and fortunately this map is epi (Theorem \ref{t1}); the proof involves a generalisation of the theory of triangulated birational motives of \cite{birat-tri} to Verdier quotients of Voevodsky's category $\DM^\eff(k)$ by higher powers of the Tate object, in the spirit of \cite{tatefil}. The reduced homotopic module associated to $\sK_*^M\otimes \Q_l$ therefore maps to the homotopic module of continuous étale sheaves, and a comparison using an analogue of Corollary \ref{c1} b) shows that this is an isomorphism. 

\section{The $l$-adic computation}\label{s1}

Let $\bar X$ be a smooth projective geometrically irreducible variety over a field $k$, $Z=\bigcup_{i\in I} Z_i\subset \bar X$ a normal crossing divisor and $X=\bar X - Z$. For $J\subseteq I$, write $Z_J=\bigcap_{i\in J} Z_i$, in particular (by convention) $Z_\emptyset = \bar X$.

Let $(i,n)\in \Z\times \Z$. If $H^i(V,n)$ denotes continuous étale cohomology $H^i_\cont(V,\Q_l(n))$ \cite{jannsen}, the exact sequences for cohomology with supports and the reasoning of \cite[3.3]{arason} yield a spectral sequence
\begin{equation}\label{eq1}
E_1^{p,q} = \bigoplus_{|J|=d-p} H^q_{Z_J}(\bar X,n)\Rightarrow H^{p+q-d}(X,n)
\end{equation}
where $d=\dim X$, and where the $d^1$ differentials are given by Gysin maps. By purity \cite{jannsen}, we have
\begin{equation}\label{eq3}
H^q_{Z_J}(\bar X,n)\simeq H^{q-2(d-p)}(Z_J,n+p-d).
\end{equation}

This yields

\begin{prop}\label{p1} Suppose $k$ finite. Then $E_1^{p,q}=0$ unless $q\in \{2n,2n+1\}$, $d-n\le p\le d$ and $n\le d$.
\end{prop}

\begin{proof} The first condition follows from the Weil conjecture. In the second condition, the upper bound is clear, while the lower bound follows from the inequality $q-2(d-p)\ge 0$ and the first condition. For the third condition, the étale cohomological dimension of $Z_J$ is $2(d-|J|)+1=2p+1$, hence $E_1^{p,q}=0$ unless $q-2(d-p)\le 2p+1$, i.e. $q\le 2d+1$, which in turn implies $n\le d$ by the first condition.
\end{proof}

\begin{cor}\label{c1} We have\\
a)  long exact sequences
\[\dots\to E_2^{r-2n,2n}\to H^{r-d}(X,n)\to E_2^{r-2n-1,2n+1}\to E_2^{r-2n+1,2n}\to \dots\]
b) $H^i(X,n)=0$ unless  $n\le d$ and $i\ge n$, and an exact sequence
\[0\to H^n(X,n)\to  \bigoplus_{|J|=n} H^0(Z_J,0)\by{i_n} \bigoplus_{|J|=n-1} H^2(Z_J,1)\]
where $i_n$ is given by the Gysin maps in continuous étale cohomology.
\end{cor}

\begin{proof} a) is obvious from the first condition on $q$ in Proposition \ref{p1}, and b) then follows from the other conditions. Indeed, all terms in a) are $0$ if $r-2n<d-n$, i.e. if $r-d<n$. If now $r-d=n$, then the middle term is isomorphic to $\Ker(E_1^{d-n,2n}\by{d_1^{}} E_1^{d-n+1,2n})$, hence the conclusion.
\end{proof}

\begin{rk} If $|I|\le d$ and $n> |I|$, we get a sharper vanishing bound: $H^i(X,n)=0$ for $i<2n-|I|$, and an exact sequence
\[0\to H^{2n-|I|}(X,n)\to H^0(Z_I,0)\to \bigoplus_{|J|=|I|-1} H^2(Z_J,1). \]
\end{rk}

\section{Reduced homotopic modules}

We go back temporarily to a general perfect field $k$, and write $\HI$ for the category of homotopy invariant Nisnevich sheaves with transfers over $k$ \cite[Lect. 13]{mvw}. Let $\HI^\o$ be the full subcategory of $\HI$ consisting of birational sheaves \cite[Def. 2.3.1]{birat-tri}. By \cite[7.1 and Th. 7.3.1]{birat-tri}, the inclusion functor $\HI^\o\inj \HI$ has a right adjoint $\sF\mapsto \sF_\nr=R^0_\nr \sF$.

\begin{defn}\label{d1} A sheaf $\sF\in \HI$ is \emph{reduced} if $\sF_\nr=0$.
\end{defn}

\begin{lemma}\label{l1} Let $\sF\in \HI$. Then the presheaf with transfers
\[\sF_\rd=\Coker(\sF_\nr\to \sF) \]
is a reduced (Nisnevich) sheaf, and the functor $\sF\mapsto \sF_\rd$ is left adjoint to the inclusion of reduced sheaves into $\HI$. 
\end{lemma}

\begin{proof} By \cite[Lemma 2.3.2]{birat-tri} we have $H^1(X,\sF_\nr)=0$ for any smooth $X$, hence a short exact sequence
\[0\to \sF_\nr(X)\to \sF(X)\to a\sF_\rd(X)\to 0\] 
where $a\sF_\rd$ is the Nisnevich sheaf associated to $\sF_\rd$; therefore $\sF_\rd\to a\sF_\rd$ is an isomorphism of presheaves. Applying now the functor $R_\nr$ of \cite[3.1]{birat-tri} to the exact sequence $0\to \sF_\nr\to \sF\to \sF_\rd\to 0$, we get an exact triangle in $\DM^\o$
\[R_\nr(\sF_\nr[0])\to R_\nr (\sF[0])\to R_\nr(\sF_\rd[0])\by{+1}\]

But $R_\nr(\sF_\nr[0])=\sF_\nr[0]$ because $R^0_\nr \sF[0]\in \DM^\o$ by \cite[Th. 4.4.1]{birat-tri} (the part on $t$-structures). Taking the long cohomology exact sequence for the homotopy $t$-structure of $\DM^\o$, it follows that $R^0_\nr(\sF_\rd)=0$, \ie that $\sF_\rd$ is reduced. That it defines the said left adjoint is now obvious.
\end{proof}

Here is a generalisation. Recall that $\sF\in \HI^\o \iff \sF_{-1}=0$, where $(-)_{-1}$ is Voevodsky's contraction \cite[Prop. 2.5.2]{birat-tri}. 

\begin{defn}\label{d2} a) A sheaf $\sF\in \HI$ is \emph{of coniveau $< n$} if $\sF_{-n}=0$.\footnote{This terminology will be justified by Lemma \ref{l4} a).} Write $\HI_{< n}$ for the full subcategory of $\HI$ consisting of sheaves of coniveau $< n$ (so that $\HI_{< 1}=\HI^\o$).\\
b) A sheaf $\sF\in \HI$ is \emph{$n$-reduced} if its only subsheaf of coniveau $< n$ is $0$.
\end{defn}

\begin{defn}\label{d3} We write $\DM_{< n}^\eff=\DM^\eff/\DM^\eff(n)$ (so that $\DM_{< 1}^\eff\allowbreak =\DM^\o$).
\end{defn}

The same yoga as in \cite{birat-tri} (Brown representability) gives

\begin{prop}\label{p2} The localisation functor $\nu_{< n}:\DM^\eff\to \DM_{< n}^\eff$ admits a (fully faithful) right adjoint $\iota_n$, which itself admits a right adjoint $R_{< n}:C\mapsto C_{< n}$. Moreover, there are functorial exact triangles
\[\nu^{\ge n} M\to M\by{\epsilon_M} \iota_n\nu_{<n} M\by{+1}\]
where $\epsilon_M$ is the unit of the adjunction $(\nu_{<n},\iota_n)$ and, as in  \cite[(1.1)]{tatefil}, $\nu^{\ge n}M = \uHom(\Z(n),M)(n)$.
\qed
\end{prop}

The key point is:

\begin{prop}\label{p3} The homotopy $t$-structure on $\DM^\eff$ induces a $t$-structure on $\DM_{< n}^\eff$ via $\iota_n$, with heart $\HI_{< n}$.
\end{prop}

\begin{proof} For $C\in \DM^\eff$, write $C_{-n}=\uHom(\Z(n)[n],C)$. Then $C\in \iota_n\DM_{< n}^\eff$ $\iff$ $C_{-n}=0$. But this functor is $t$-exact as the $n$-fold composition of the $t$-exact functor $(-)_{-1}$ \cite[Th. 5.2]{deglise}.
\end{proof}

\begin{prop}\label{p4} a) The inclusion $\HI_{< n}\inj \HI$ has a right adjoint $\sF\mapsto \sF_{< n}$. Moreover, we have $(\sF_{-1})_{< n-1}=(\sF_{< n})_{-1}$ as subsheaves of $\sF_{-1}$.\\
b) The inclusion of $n$-reduced sheaves in $\HI$ has a left adjoint $\sF\mapsto \sF_\nrd$, and the unit morphism $\sF\to \sF_\nrd$ is an epimorphism of sheaves.
\end{prop}

\begin{proof} Since $(-)_{-n}$ is exact and commutes with infinite direct sums, $\HI_{< n}$ is stable under arbitrary colimits; defining $\sF_{< n}=\colim \sG$, where $\sG$ runs through the subsheaves of $\sF$ which belong to $\HI_{< n}$, proves the first part of a). For the second part, the exactness of $(-)_{-1}$ gives an inclusion $(\sF_{< n})_{-1}\subseteq (\sF_{-1})_{< n-1}$; conversely, the inclusion $(\sF_{-1})_{< n-1}\subseteq \sF_{-1}$ yields by adjunction a morphism $(\sF_{-1})_{< n-1}\otimes \G_m\to \sF$, which factors through $\sF_{< n}$ by the cancellation theorem \cite{voecan}, hence $(\sF_{-1})_{< n-1}\subseteq (\sF_{< n})_{-1}$ by adjunction again.

For b), define $\sF_\nrd=\Coker(\sF_{< n}\to \sF)$. Using Proposition \ref{p3}, the same reasoning as in the proof of Lemma \ref{l1} shows that $\sF_\nrd$ is $n$-reduced, hence defines the desired left adjoint.
\end{proof}

\begin{rk} Contrary to Lemma \ref{l1}, the map $\sF\to \sF_\nrd$ may not be an epimorphism of presheaves if $n>1$.
\end{rk}

Recall from \cite{deglise} that a \emph{homotopic module} is an $\Omega$-$\G_m$-spectrum in $\HI$, \ie a sequence $(\sF_n)_{n\in \Z}$ of objects of $\HI$ provided with isomorphisms $\sF_n\iso (\sF_{n+1})_{-1}$. We shall say that a homotopic module $(\sF_n)$ is \emph{$(-1)$-connected} if $\sF_n=0$ for $n<0$. Write $\HI_*$ for the category of homotopic modules, and $\HI_*^c$ for its full subcategory of $(-1)$-connected homotopic modules.

\begin{defn}\label{d4} A $(-1)$-connected homotopic module $(\sF_n)_{n\ge 0}$ is \emph{reduced} if $\sF_n$ is reduced for all $n>0$. Write $\HI_*^\rd$ for the full subcategory of $\HI_*^c$ formed of reduced homotopic modules.
\end{defn}

\begin{lemma}\label{l2} If $(\sG_n)\in \HI_*^c$ is reduced, then $\sG_n$ is $n$-reduced for all $n\ge 0$.
\end{lemma}

\begin{proof} Induction on $n$. The case $n= 0$ is trivial. Suppose the statement true for $n-1\ge 0$, and let $\sH\subseteq  \sG_n$ with $\sH\in \HI_{< n}$. Then $\sH_{-1}\subseteq (\sG_n)_{-1}$ is $0$ since $\sH_{-1}\in \HI_{< n-1}$. As $\sG_n$ is reduced, we have $\sH=0$.
\end{proof}

\begin{thm}\label{t1} The inclusion $\HI_*^\rd\inj \HI_*^c$ has a left adjoint $(\sF_*)\mapsto (\sF_*)^\rd$. The unit of this adjunction is an epimorphism of graded sheaves.
\end{thm}

\begin{proof} For $(\sF_*)\in \HI_*^c$ and $n\ge 0$, define
\[ \sF_n^\rd = (\sF_n)_\nrd.\]

By Proposition \ref{p4} a), the isomorphisms $\sF_{n-1}\iso (\sF_n)_{-1}$ induce isomorphisms $(\sF_{n-1})_{< n-1}\iso ((\sF_n)_{< n})_{-1}$, hence isomorphisms $\sF_{n-1}^\rd\iso (\sF_n^\rd)_{-1}$ by Proposition \ref{p4} b). Thus $(\sF_*)^\rd:=(\sF_*^\rd)\in \HI_*^c$, and this homotopic module is reduced. Its universal property now follows from Lemma \ref{l2}.
\end{proof}





\section{Cohomology}

\begin{lemma}\label{l3} Let $\sF\in \HI$. If $F$ is a smooth closed subset of pure codimension $c$ in a smooth $k$-scheme $X$, there are isomorphisms $H^i_F(X,\sF)\allowbreak\simeq H^{i-c}(F,\sF_{-c})$, hence a long exact sequence for $U=X-F$:
\[\dots\to H^{i-c}(F,\sF_{-c}) \to H^i(X,\sF)\to H^i(U,\sF)\by{\partial} H^{i+1-c}(F,\sF_{-c}) \to\dots\]
In particular, we have $\sF(X)\iso \sF(U)$ if $c>1$, and an exact sequence
\[0\to \sF(X)\to \sF(U)\to \sF_{-1}(F)\]
if $c=1$.
\end{lemma}

\begin{proof} Let $M_F(X)=cone(M(U)\to M(X))$. We have a Gysin isomorphism $M_F(X)\simeq M(F)(c)[2c]\simeq M(F)\otimes \G_m^{\otimes c}[c]$ \cite[Th. 15.15]{mvw}, hence
\begin{multline*}
H^i_F(X,\sF)\simeq \DM^\eff(M_F(X),\sF[i])\simeq \DM^\eff(M(F)\otimes \G_m^{\otimes c}[c],\sF[i])\\
\simeq \DM^\eff(M(F),\sF_{-c}[i-c])\simeq H^{i-c}(F,\sF_{-c}).
\end{multline*}
\end{proof}

\begin{prop}\label{p6} Let $p:Y\to X$ be an alteration of smooth $k$-schemes. Then, for any $\sF\in \HI$, there exists a map $p_*:\sF(Y)\to \sF(X)$, natural in $\sF$, such that the composition $\sF(X)\by{p^*} \sF(Y)\by{p_*} \sF(X)$ is multiplication by the generic degree $\delta$.
\end{prop}

\begin{proof} If $p$ is finite, this follows from the transfer structure on $\sF$. In general, let $Y\by{q}Z\by{r} X$ be the Stein factorisation of $p$. Considering the normalisation of $Z$, we see that $Z$ is normal. Therefore, by the valuative criterion of properness there exists a closed subset $F\subset Z$ of codimension $\ge 2$ such that $q$ is an isomorphism above $Z-F$. Then $F'=r(F)$ is of codimension $\ge 2$ in $X$, and $F''=r^{-1}(F')$ is still of codimension $\ge 2$ in $Z$. Let $G=q^{-1}(F'')$; then $q_{|Y-G}:Y-G\to Z-F''$ is an isomorphism, hence $p':=p_{|Y-G}:Y-G\to X-F'$ is finite. We define $p_*$ as the composition
\[\sF(Y)\to \sF(Y-G)\by{p'_*} \sF(X-F')\iso \sF(X)\]
where the last map is the inverse of the isomorphism of Lemma \ref{l3}. Using the commutative diagram
\begin{equation}\label{eq6}
\begin{CD}
\sF(Y)@>>>\sF(Y-G)\\
@A p^* AA @A{p'}^*AA\\
\sF(X)@>\sim>>\sF(X-F')
\end{CD}
\end{equation}
we see that $p_*p^*$ is multiplication by $\delta$.
\end{proof}

\begin{rk} If $p$ is birational, ${p'}^*$ is an isomorphism in \eqref{eq6}. Since its top map is injective (Lemma \ref{l3}), $p^*$ is an isomorphism. 
\end{rk}

The following lemma will not be used in the sequel, but seems worth noting. It generalises \cite[Lemma 2.3.2]{birat-tri}, which is its special case $n=1$.

\begin{lemma}\label{l4} Let $\sF\in \HI_{< n}$. Then\\
a) If $Z\subset X$ is a closed pair of smooth varieties, with $Z$ of codimension $\ge n$, then $H^*(X,\sF)\iso H^*(X-Z,\sF)$.\\
b) $H^i(X,\sF)=0$ for $i\ge n$ and any smooth $X$.
\end{lemma}

\begin{proof} a) follows from Lemma \ref{l3} and the definition of $\HI_{<n}$. For b), by induction on $n$ the first and last group in the exact sequence of this lemma are $0$ for $i-c\ge n-c$, hence for $i\ge n$.  By a standard argument of successive singular loci, this implies that $H^i(X,\sF)\iso H^i(U,\sF)$ when $i\ge n$ for any open immersion $U\inj X$; but the functor $X\mapsto H^i(X,\sF)$ is effaceable for $i>0$ in the sense that every cohomology class vanishes locally for the Zariski topology, hence the conclusion.
\end{proof}

\begin{rk} The above results can be deduced more elementarily from Voevodsky's Gersten resolution \cite[Th. 4.37]{voepre}.
\end{rk}

\begin{thm}\label{t3} Let $n>0$. For any smooth projective variety $X$, the counit map of the adjunction of Proposition \ref{p3}
\[\iota_n\nu_{<n} M(X)\to M(X)\]
becomes an isomorphism after applying the truncation functor $\tau_{\le -n}$ (cohomological notation).
\end{thm}

\begin{proof} It suffices to show that $\tau_{> -n} \nu^{\ge n} M(X)=0$. Writing
\[ \nu^{\ge n} M(X) = \uHom(\Z(n)[n],M(X))(n)[n]= \uHom(\Z(n)[n],M(X))\otimes \G_m^{\otimes n}\]
and noting that tensor product is right $t$-exact in $\DM^\eff$, it suffices to show that $\tau_{> -n} \uHom(\Z(n)[n],M(X))=0$. This was proven in \cite[Prop. 2.3]{rqnr}. 
\end{proof}

\begin{cor}\label{c2} Let $X$ be a smooth projective variety. Then, for any  $\sF\in \HI$, the counit map
\[H^i(X,\iota_n R_{<n} \sF)\to H^i(X,\sF)\]
is an isomorphism for $i<n$. If $\sF\in \HI_\nrd$, both sides are $0$ for $i=0$.
\end{cor}

\begin{proof} The first point follows directly from Theorem \ref{t3}. The second follows from the first, since $R^0_{<n} \sF=\sF_{<n}$ for any $\sF$. 
\end{proof}

Let $\sF$ be any abelian sheaf on the big Zariski site on smooth $k$-varieties. For $(\bar X,Z,X)$ as in the beginning of \S \ref{s1}, we have a spectral sequence similar to \eqref{eq1}:
\begin{equation}\label{eq2}
E_1^{p,q} = \bigoplus_{|J|=d-p} H^q_{Z_J}(\bar X,\sF)\Rightarrow H^{p+q-d}(X,\sF).
\end{equation}

If $\sF=\sF_n$ is part of a homotopic module, Lemma \ref{l3} yields this time
\begin{equation}\label{eq4}
H^q_{Z_J}(\bar X,\sF_n)\simeq H^{p+q-d}(Z_J,\sF_{n+p-d}).
\end{equation}

\begin{prop}\label{t2}  Suppose that $(\sF_n)\in \HI_*^\rd$ (see Definition \ref{d4}). Then, for $p+q=d$,  we have $E_1^{p,q}=0$ except for $p=d-n$, hence an exact sequence
\[0\to \sF_n(X)\to \bigoplus_{|J|=n} \sF_0(Z_J)\by{i_n} \bigoplus_{|J|=n-1} H^1(Z_J,\sF_1) \]
where $i_n$ is induced by the boundary maps $\partial$ of Lemma \ref{l3}.
\end{prop}

\begin{proof} The first claim follows from Lemma \ref{l2} and Corollary \ref{c2}, and the second follows from the first.
\end{proof}

\section{Back to Section \ref{s1}}

We now make the link with the situation in that section, so assume again $k$ finite. For any smooth $k$-scheme $X$, write
\[\sH_n(X)=H^n_\cont(X,\Z_l(n)).\]

\begin{prop}\label{p5} The presheaf $\sH_n$ has a transfer structure and is $\A^1$-inv\-ar\-iant; after tensoring with $\Q$, it becomes an étale sheaf, and the collection $(\sH_n\otimes \Q)_{n\in \Z}$ is an object of $\HI_*^\rd$.
\end{prop}

\begin{proof} That finite correspondences act on étale cohomology with coefficients in twisted roots of unity follows from \cite[Th. 10.3]{mvw}. Since this action commutes with change of coefficients, it induces one on  $\sH_n$. Its $\A^1$-invariance is classical, and moreover $(\sH_n)_{-1}\simeq \sH_{n-1}$ by the projective line formula in étale cohomology.  With the notation of the introduction, $H^n(\Q_l(n)^c)$ is the étale sheaf associated to $\sH_n\otimes \Q$, which is therefore already an étale sheaf by \eqref{eq0}. Moreover, the Weil conjectures imply that $\Q_l(n)^c=0$ for $n<0$ \cite[Cor. 6.10 b)]{tate}, hence $(\sH_n\otimes \Q)\in \HI_*^c$. Finally,  $\sH_n\otimes \Q$ is reduced for $n>0$ once again by the Weil conjectures plus de Jong's theorem \cite[Th. 4.1]{dJ}, since $\sF_{<1}(X) =\sF(\bar X) $ for any $\sF\in \HI$ if $X$ has a smooth compactification $\bar X$ \cite[Cor. 7.3.2]{birat-tri}. 
\end{proof}


Since $(\sK_n^M)\in \HI_*^c$, \eqref{eq-1} for all $n\ge 0$ factors through a morphism in $\HI_*^\rd$
\begin{equation}\label{eq5}
(\sK_n^M)\otimes \Q_l\surj (\sK_n^M)^\rd\otimes \Q_l\to (\sH_n\otimes \Q) 
\end{equation}
by Theorem \ref{t1} and Proposition \ref{p5}. Here we forget that $\sH_n\otimes \Q$ is an étale sheaf, and only remember its Nisnevich sheaf structure.

\begin{thm}\label{t4} Let $(\bar X,Z,X)$ be as in the beginning of \S \ref{s1}. Then the second map of \eqref{eq5} is an isomorphism when evaluated at $X$.
\end{thm}

\begin{proof} By functoriality,  Propositions \ref{p1} and \ref{t2} yield via \eqref{eq5} a commutative diagram of short exact sequences
\[\begin{CD}
0@>>> \sF_n(X)@>>> \bigoplus_{|J|=n} \sF_0(Z_J)@>{i_n}>> \bigoplus_{|J|=n-1} H^1(Z_J,\sF_1)\\
&&@Va VV @Vb VV @Vc VV\\
0@>>> \sH_n(X)@>>> \bigoplus_{|J|=n} H^0(Z_J,0)@>{i_n}>> \bigoplus_{|J|=n-1} H^2(Z_J,1)
\end{CD}\]
where $\sF_n$ is the $n$-th term of $(\sK_n^M)^\rd\otimes \Q_l$. But $\sK_0^M=\Z$, hence  $b$ is an isomorphism, and $\sK_1^M=\G_m$; in particular, $\sK_1^M\otimes \Q_l$ is reduced because $E^*$ is finite for any finite extension $E/k$. It follows that $c$, which is the cycle class map for divisors, is injective. A diagram chase now shows that $a$ is bijective.
\end{proof}

\section{Proof of Theorem \ref{t0}}

The field $k$ is still finite.

\begin{thm} \label{t5} The second map of \eqref{eq5} is an isomorphism for any $n\in \Z$. If $n\le 2$, the composition is an epimorphism of presheaves.
\end{thm}

\begin{proof} Let $X$ be smooth irreducible. By \cite[Th. 4.1]{dJ}, applied with $Z=\emptyset$, there is an alteration $p:X_1\to X$ and a dense open immersion $X_1\subseteq  \bar X_1$ such that $\bar X_1$ is smooth projective and $\bar X_1-X_1$ is the support of a divisor with strict normal crossings. By Theorem \ref{t4}, the statement is true at $X_1$, hence it is true at $X$ thanks to Proposition \ref{p6}. For $n\le 2$, the claim follows from Lemma \ref{l1}.
\end{proof}

\begin{thm}\label{t0Q} 
The morphism \eqref{eq-4} is an epimorphism after tensorisation with $\Q$. 
\end{thm} 

\begin{proof}
By Theorem \ref{t5} and the epimorphy in Theorem \ref{t1},  \eqref{eq-1} is an epimorphism of Nisnevich sheaves, hence also of Zariski sheaves by \cite[Th. 22.2]{mvw}. 
\end{proof}

\begin{prop}\label{p7} The cokernel of $H^n(X,\Z(n))\otimes \Z_l\to H^n_\cont(X,\Z_l(n))$ is torsion-free, and its kernel is divisible.
\end{prop}

\begin{proof} By \cite[Cor. 3.5]{cycle-etale}, the map $H^i_\et(X,\Z(n))\otimes \Z_l\to H^i_\cont(X,\Z_l(n))$ has the said properties for any $i$. On the other hand,  the norm residue isomorphism theorem implies that the map $H^n(X,\Z(n))\otimes \Z_l\to H^n_\et(X,\Z(n))\allowbreak\otimes \Z_l$ is an isomorphism (Beilinson-Lichtenbaum conjecture, \cite{voem}).
\end{proof}

\begin{proof}[Proof of Theorem \ref{t0}] By Theorem \ref{t0Q}, the cokernel of \eqref{eq-4} is torsion. On the other hand,  that of Proposition \ref{p7} remains torsion-free after Zariski-sheafification (an exact functor). But  \eqref{eq-4} factors as a composition
\begin{equation}\label{eq6.2}
\sK_n^M\otimes \Z_l\to \sH^n(\Z(n))\otimes \Z_l\to \sH^n_\cont(\Z_l(n))
\end{equation}
where the first morphism is epi by \cite[Th. 5.1]{mvw} and \cite[Prop. 4.3]{ev-ms}. Therefore the cokernel of \eqref{eq-4} is $0$.
\end{proof}

\begin{rk} In \eqref{eq6.2}, the kernel of the second morphism is divisible (same reasoning as before), while the kernel of the first morphism is killed by $(n-1)!$ (see \emph{e.g.} parts (6) and (11) of \cite[Prop. 10]{kerz}), where its failure to be $0$ comes from too small residue fields (Part (5) of \emph{loc. cit.}). To correct this and obtain divisibility of  the kernel of \eqref{eq-4}, one may replace the  sheaves of Milnor $K$-groups by sheaves of improved Milnor $K$-groups as in \emph{loc. cit.}
\end{rk}







\section{The global sections of Milnor $K$-sheaves}\label{s7}

To say that the sheaf $\sK_n^M\otimes \Q$ is reduced is exactly to say that 
\begin{equation}\label{eq7}
H^0(X,\sK_n^M)\otimes \Q=0
\end{equation}
for any connected smooth projective $k$-variety $X$. This is true for $n=1$, because this group is $E^*\otimes \Q$ where $E$ is the field of constants of $X$, and $E$ is finite (this fact was used in the proof of Theorem \ref{t4}). For $n>1$ it is open, but still true for certain smooth projective $X$:
recall that $X$ is \emph{of abelian type} if its Chow motive is a direct summand of that of an abelian variety (possibly after a finite extension of $k$). Then 

\begin{thm}\label{t6} Let $n\ge 2$. Suppose that $X$ is of abelian type and that the Tate conjecture holds for $X$ in codimension $n$. Then  \eqref{eq7} holds. 
\end{thm}

\begin{proof}  It is analogous to that of \cite[Lemma 1.6]{cell} or \cite[Th. 5.4]{ladic}, so we only sketch it. 
We have
\begin{equation}\label{eq8}
H^0(X,\sK_n^M)\simeq \DM^\eff(M(X),\sK_n^M[0)).
\end{equation}

Write $\sM_\rat$ (resp. $\sM_\num$) for the category of pure motives over $k$ with rational coefficients modulo rational (resp. numerical) equivalence.  Let $\bigoplus_{i\in I} S_i$ be a decomposition of 
$h_\num(X)\in \sM_\num$ into a direct sum of simple motives. By 
Kimura's nilpotence theorem \cite[Prop. 7.5 and Ex. 9.1]{kim}, lift this decomposition to an isomorphism $h_\rat(X)\simeq \bigoplus_{i\in I} \tilde S_i$ in $\sM_\rat$. If 
$\Phi:\sM_\rat^\eff\to \DM^\eff$ 
is the natural functor, we thus have
\[
M(X) \simeq \bigoplus_{i\in I} \Phi(\tilde S_i).\]

On the direct summand 
$\DM^\eff(\Phi(\tilde S_i),\sK_n^M[0))$ of the right hand side of \eqref{eq8}, the action of Frobenius on the left and right term of the Hom induces the same action on the Hom, by naturality. Its action on 
$\sK_n^M[0)$ is multiplication by $q^n$ (where $q=|k|$), while its action on $\Phi(\tilde S_i)$ is killed by a suitable power of the minimal polynomial $\Pi_i$ of the Frobenius endomorphism of $S_i$. Therefore, if $\Pi_i(q^n)\ne 0$ then this direct summand is torsion.

The remaining case is the one which involves the Tate conjecture. Namely, suppose that $\Pi_i = T-q^n$. Then, inside $H^{2n}(X_{\bar k},\Q_l)$, the geometric Frobenius acts on the summand $H^{2n}((\tilde S_i)_{\bar k},\Q_l)$ by multiplication by $q^n$. By the Tate conjecture this corresponds to an element of $CH^n(X)\otimes \Q_l$, hence to a nonzero morphism $\tilde S_i\to \L^n$; by Schur's lemma it is an isomorphism modulo numerical equivalence, hence also modulo rational equivalence again by Kimura nilpotence. But $\Phi(\L^n)=\G_m^{\otimes n}[n]$, hence $\DM^\eff(\Phi(\tilde S_i),\sK_n^M[0])=0$ again.
\end{proof}

Since the Tate conjecture obviously holds if $n\ge \dim X$, we get:

\begin{cor}\label{c3} Let $X$ be a smooth $k$-surface which is birational to a smooth projective surface of abelian type. Then the map
\[H^0(X,K_n^M)\otimes \Q_l\to H^n_\cont(X,\Q_l(n))\]
induced by \eqref{eq-1} is bijective for any $n\ge 0$.
\end{cor}

\begin{proof} By Abhyankar resolution (and embedded resolution of curves), we are in the situation at the beginning of §\ref{s1}; moreover, since smooth projective curves are of abelian type, to be of abelian type is a birational invariant of smooth projective surfaces, so that $\bar X$ and the $Z_i$ are of abelian type in loc. cit. By Theorem \ref{t6}, we may therefore run the proof of Theorem \ref{t4} by taking $\sF_n=\sK_n^M\otimes \Q_l$ instead of  $(\sK_n^M)^\rd\otimes \Q_l$.
\end{proof}

\begin{ex}\label{ex1} A smooth projective surface such that $b_2=\rho$ is of abelian type if and only if it verifies Bloch's conjecture (e.g. if it is not of general type). Fermat surfaces are of abelian type, etc.
\end{ex}

\enlargethispage*{30pt}

\end{document}